\documentclass[11pt,a4paper]{article}

\usepackage{epsf,epsfig,amsfonts,amsgen,amsmath,amstext,amsbsy,amsopn,amsthm,cases,listings,color
}
\usepackage{ebezier,eepic}
\usepackage{color}
\usepackage{multirow}
\usepackage{epstopdf}
\usepackage{graphicx}
\usepackage{pgf,tikz}
\usepackage{mathrsfs}
\usepackage[marginal]{footmisc}
\usepackage{enumitem}
\usepackage[titletoc]{appendix}
\usepackage{booktabs}
\usepackage{url}
\usepackage{mathtools}
\usetikzlibrary{positioning, shapes.geometric, fit}
\usepackage{caption}
\usepackage{subcaption}
\usetikzlibrary{positioning}
\usepackage{pgfplots}
\usepackage{authblk}
\usepackage{amssymb}
\usepackage{circuitikz}
\usepackage{wasysym}
\usepackage{comment}
\usepackage{empheq}
\usetikzlibrary{shapes.geometric, positioning, matrix, calc}
\usepackage{dsfont}
\usepackage{indentfirst}

\usepackage{caption}
\usepackage{subcaption}
\pgfplotsset{compat=1.18}
\usepackage{mathrsfs}
\usepackage{wasysym} 
\usetikzlibrary{arrows}
\usepackage{aligned-overset}
\usepackage{bm}
\usepackage[backref=page]{hyperref}
\usepackage{makecell}


\allowdisplaybreaks[1]

\definecolor{uuuuuu}{rgb}{0.27,0.27,0.27}
\definecolor{sqsqsq}{rgb}{0.1255,0.1255,0.1255}

\newtheorem{definition}{Definition} [section]
\newtheorem{theorem}[definition]{Theorem}
\newtheorem{lemma}[definition]{Lemma}
\newtheorem{proposition}[definition]{Proposition}

\newtheorem{claim}[definition]{Claim}
\newtheorem{problem}[definition]{Problem}

\newenvironment{pf}{\noindent{\bf Proof}}

\begin{document}
\title{\bf\Large  On saturation problems involving clique number and matching number}
\date{}
\author[1]{Zian Chen\thanks{Email: \texttt{michaelchen24@163.com}}}
\author[1,2]{Guorong Gao \thanks{
Research supported by National Key R\&D Program of China (Grant No. 2023YFA1010202),
National Natural Science Youth Foundation of China (Grant No. 12401448), Natural Science Foundation of Fujian Province (Grant No. 2024J08030).
Email: \texttt{grgao@fzu.edu.cn}}}  
\author[1]{Jianfeng Hou\thanks{Research supported by National Key R$\&$D Program of China (Grant No. 2023YFA1010202). Email: \texttt{jfhou@fzu.edu.cn}}}

\author[3]{Yue Ma\thanks{Research supported by National Natural Science Foundation of China (Grant No. 12401455 and Grant No. 12571369). Email: \texttt{yma@njust.edu.cn}}}
\affil[1]{Center for Discrete Mathematics,
            Fuzhou University, Fuzhou,  Fujian, China}
\affil[2]{School of Mathematics and Statistics,
 Fuzhou University, Fuzhou, Fujian, China}
\affil[3]{School of Mathematics and Statistics,
Nanjing University of Science and Technology,
Nanjing, Jiangsu, China.}
\maketitle
\begin{abstract}
For a clique $K_r$, a graph is $K_r$-saturated if it contains no copy of $K_r$ and the addition of any edge from its complement creates a $K_r$.
A classical result of Erd\H{o}s-Hajnal-Moon and Zykov shows that the number of edges of an $n$-vertex $K_r$-saturated graph is at least $(r-2)n-\binom{r-1}{2}$. In this paper, 
we focus on the number of edges of the $K_r$-saturated graphs with a fixed matching number. Let $G$ be an $n$-vertex $K_r$-saturated graph with matching number $\nu(G) = s$. For sufficiently large $n$, we prove that the number of edges
\begin{equation*}
     e(G)\geq \left\{\begin{array}{cl}{(r-1)n-\frac{r}{2}(r-1)-1,}&{\quad\mathrm{if}~s=r-1;}\\{(r-1)n + (s-r)^2 - \frac{1}{2}(r+2)(r-3) - 5,}&{\quad\mathrm{if}~s>r-1.}\\\end{array}\right.
\end{equation*}
Moreover, we completely characterize the graphs attaining the equality.
\medskip

\end{abstract}
\section{Introduction}\label{SEC:Intorduction}

In this paper, all graphs considered are finite, undirected, and simple. Let $G=(V(G),E(G))$ be a graph with vertex set $V(G)$ and edge set $E(G)$. We denote the number of edges in $G$ by $e(G)$, and let $n = |V(G)|$ denote the number of vertices. 
For standard graph-theoretic notation, we use $K_r$ to denote the complete graph on $r$ vertices, and $P_l$ and $C_l$ to denote the path and cycle of order $l$, respectively. 
A \textit{matching} in a graph is a set of edges without common vertices. The \textit{matching number}, denoted by $\nu(G)$, is the maximum size of a matching in $G$. 

A graph $G$ is said to be \textit{$H$-saturated} if $G$ does not contain $H$ as a subgraph, but the addition of any edge $xy \in E(\bar{G})$ creates a copy of $H$. 
The classical saturation number, denoted by $sat(n, H)$, is the minimum number of edges in a $H$-saturated graph on $n$ vertices. Saturation problems serve as a minimization counterpart to the classical Tur\'{a}n problems. While the Tur\'{a}n number $ex(n, K_r)$ asks for the maximum number of edges in a $K_r$-free graph, the saturation number $sat(n, K_r)$ seeks the minimum. 
The systematic study of saturation numbers was initiated independently by Erd\H{o}s-Hajnal-Moon \cite{EHM64} and Zykov \cite{Z52} in the 1960s. They showed that
\[
sat(n, K_r) = (r-2)n - \binom{r-1}{2},
\]
for $n \ge r-1$ and the unique extremal graph is $K_{r-2} \vee \bar{K}_{n-r+2}$, where the extremal graph is a saturated graph attaining the saturation number. 
A remarkable distinction between these two parameters, observed by K\'{a}szonyi-Tuza \cite{KT86}, is that $sat(n, F) = O(n)$ for any fixed graph $F$, whereas $ex(n, F)$ is typically $O(n^2)$. For a comprehensive survey on saturation numbers, we refer the reader to Faudree-Faudree-Gould-Jacobson \cite{FFGJ11}.

Following this foundational work, researchers began investigating saturation numbers under additional structural constraints. In 1965, Hajnal \cite{H65} studied $K_r$-saturated graphs that do not contain a dominating vertex (a vertex of degree $n-1$).
Later, F\"{u}redi-Seress \cite{FS94} and Amin et al. \cite{AFGS13} extended these results by determining $sat(n, K_r)$ for graphs with specific maximum degree constraints.
In a related direction, Erd\H{o}s-Holzman \cite{EH94} analyzed the edge number of $K_3$-saturated graphs with maximum degree $cn$. They showed that as $n\rightarrow\infty$, the saturation number of $K_3$ under this constrain is asymptomatically equal to $\frac{1}{2}(11-7c)n+o(n)$ for $\frac{3}{7} \le c \le \frac{1}{2}$, and $4n+o(n)$ for $\frac{2}{5} \le c \le \frac{3}{7}$.

For further results concerning limits on the minimum degree, we refer readers to \cite{D17}  and  \cite{DH86}.

Recently, the interplay between classical extremal parameters and matching numbers has attracted significant attention. In 2024, Alon and Frankl \cite{AF24} initiated the investigation of Tur\'{a}n numbers for graphs with bounded matching number. 
Inspired by their work, and by the recent study of saturation numbers for matchings by Zhang, Lu, and Yu \cite{ZLY24}, we investigate saturation numbers for cliques under a prescribed matching number.

Let $SAT(n,r,s)$ denote the family of $n$-vertex graphs that are $K_r$-saturated and satisfy $\nu(G) = s$. We define the constrained saturation number as:
\[
sat(n,r,s) = \min\{e(G) : G \in SAT(n,r,s)\}.
\]
For $n>2r-4$, every $K_r$-saturated graph $G$ satisfies
$\nu(G)\ge r-2.$
Indeed, since every $K_r$-saturated graph has minimum degree at least $r-2$, by greedily selecting disjoint edges, one obtains a matching of size $r-2$.
When $\nu(G)=r-2$, by the theorem of Erd\H{o}s-Hajnal-Moon \cite{EHM64}, we have $sat(n,r,r-2)=sat(n, K_r) = (r-2)n - \binom{r-1}{2}$.
Our main result gives the exact value of $sat(n,r,s)$ for $s\geq r-1$.

Before stating our main result, we describe the extremal constructions. Let $C_5$ be a cycle of length five with vertex set $\{v_1,v_2,v_3,v_4,v_5\}$ and edge set $\{v_iv_{i+1}: 1\leq i\leq 5\}$, where indices are taken modulo $5$. A blow-up of $C_5$, denoted by $C_5(t_1,t_2,t_3,t_4,t_5)$, is the graph with vertex set $\bigcup_{1\leq i\leq 5} V_i$ where $|V_i|=t_i$, and edge set 
$$\{uv: u\in V_i \text{ and } v\in V_{i+1} \text{ for }1\leq i\leq 5\},$$
Let the graph $$G(n,3, s)=C_5(1,n-2s+2,1,s-2,s-2)$$ 
as shown in Figure \ref{fig:both_graphs} and the graph family
$$\mathcal{H}(n,3)=\{C_5(1,t,1,n-t-3,1): 1\leq t\leq n-4\}.$$
For two graphs $G$ and $H$ with disjoint vertex sets, the \emph{union} of $G$ and $H$, denoted by $G\cup H$, is the graph with vertex set $V(G)\cup V(H)$ and edge set $E(G)\cup E(H)$.
The \emph{join} of $G$ and $H$, denoted by $G\vee H$, is the graph union $G\cup H$ together with all the edges joining $V(G)$ and $V(H)$.
For $s\geq r > 3$, we define $$G(n,r,s) = K_{r-3} \vee G(n-r+3,3, s-r+3)$$ and 
$$\mathcal{H}(n,s)=\{K_{r-3} \vee H: H\in \mathcal{H}(n-r+3,3)\}.$$
\begin{figure}[htbp]
    \centering
    \begin{subfigure}{0.45\textwidth}
\begin{circuitikz}
scale=0.5

\tikzstyle{every node}=[font=\fontsize{18.2pt}{23.7pt}\selectfont]
\fill[black] (5,15.125) circle (0.125cm);
\draw  (3.5,13.625) ellipse (1.6cm and 0.5cm);
\draw  (3.5,11.625) ellipse (1.6cm and 0.5cm);
\fill[black]  (5,9.875) circle (0.125cm);
\draw  (7.5,12.75) ellipse (1.8cm and 0.75cm);
\node [font=\fontsize{18.2pt}{23.7pt}\selectfont, fill={rgb,255:red,255; green,255; blue,255}, fill opacity=1, text opacity=1, inner xsep=0.080cm, inner ysep=0.085cm, rounded corners=0.020cm] at (3.5,13.625) {$s-2$};
\node [font=\fontsize{18.2pt}{23.7pt}\selectfont, fill={rgb,255:red,255; green,255; blue,255}, fill opacity=1, text opacity=1, inner xsep=0.080cm, inner ysep=0.085cm, rounded corners=0.020cm] at (3.5,11.625) {$s-2$};
\node [font=\fontsize{18.2pt}{23.7pt}\selectfont, fill={rgb,255:red,255; green,255; blue,255}, fill opacity=1, text opacity=1, inner xsep=0.080cm, inner ysep=0.085cm, rounded corners=0.020cm] at (7.5,12.75) {$n-2s+2$};
\draw (5,15.125) to[short] (3.5,14.125);
\draw (5,15.125) to[short] (2.125,13.875);
\draw (5,15.125) to[short] (4.875,13.875);
\draw [short] (3.375,13.125) -- (3.375,12.125);
\draw [short] (2.125,13.375) -- (2.125,11.875);
\draw [short] (4.875,13.375) -- (4.875,11.875);
\draw [short] (3.5,11.125) -- (5,9.875);
\draw [short] (5,9.875) -- (2.125,11.375);
\draw [short] (5,9.875) -- (4.875,11.375);
\draw [short] (5,15.125) -- (7.625,13.5);
\draw [short] (5,9.875) -- (7.75,12);
\end{circuitikz}
    \end{subfigure}

 \caption{The structure of $G(n,3,s)$}
    \label{fig:both_graphs}
\end{figure}
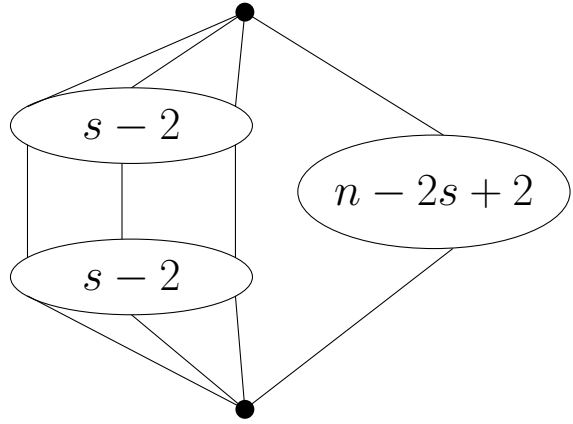

\begin{theorem}\label{thm:rs}
    Let  $s+1 \geq  r \ge 3$ and $n\geq 63(s-2)^2+r^2$. Then
\begin{equation*}
     sat(n,r,s)=\left\{\begin{array}{cl}{(r-1)n-\frac{r}{2}(r-1)-1,}&{\quad\mathrm{if}~s=r-1;}\\{(r-1)n + (s-r)^2 - \frac{1}{2}(r+2)(r-3) - 5,}&{\quad\mathrm{if}~s>r-1.}\\\end{array}\right.
\end{equation*}
    Moreover, if $G$ is an $n$-vertex $K_r$-saturated graph with matching number $s$ and $sat(n,r,s)$ edges, then $G\cong K_{r-3}\vee K_{2,n-r+1}$ for $s=r-1$, $G\in \mathcal{H}(n,s)$ for $s=r$ and $G\cong G(n,r,s)$ for $s> r$.
\end{theorem}

The rest of this paper is organized as follows. In the next section, we introduce some notations and previous results required for our proof. Section \ref{SEC:prelim} contains the detailed proofs. Finally, we conclude the paper with a discussion of a further problem.

\section{Preliminaries}\label{SEC:Preliminary}
In this section, we present some definitions and preliminary results. Let $G$ be a graph and $v\in V(G)$.
We use $x\sim y$ to denote $\{x,y\}\in E(G)$ and $x\not\sim y$ to denote $\{x,y\}\notin E(G)$. For a vertex set $X\subseteq V(G)$, the neighborhood of $v$ in $X$, denoted by $N_X(v)$, is a vertex set consisting of the vertices in $X$ that are adjacent to $v$. Let $d_X(v)=|N_X(v)|$ be the degree of $v$ in $X$. 
For two sets $A,B\subseteq V(G)$, denote the edges with two ends in $A$ by $E(A)$, and the edges between $A$ and $B$ by $E(A,B)$. Let $e(A)=|E(A)|$ and $e(A,B)=|E(A,B)|$.
We use $\Delta(G)$ and $\delta(G)$ to denote the maximum degree and minimum degree, respectively.  

\begin{theorem}[Hajnal \cite{H65}]\label{thm:rdegree}
    If $G$ is an n-vertex  $K_r$-saturated graph with $\Delta(G) < n-1$, then the minimum degree satisfies $\delta(G) \ge 2(r-2)$.
\end{theorem}

In 1996, Alon, Erd\H{o}s, Holzman, and Krivelevich studied saturation problems with bounded degree.
\begin{theorem}[Alon et al.  \cite{AEHK96}]\label{thm:4degree}
   Let $G$ be an $n$-vertex $K_4$-saturated graph with $\delta(G)=4$. If $G$ contains no vertex of degree $n-1$, then $e(G)\ge 4n-15$.
\end{theorem}

A vertex cover of a graph $G$ is a vertex set $S$ such that every edge of $G$ has at least one endpoint in $S$. A minimum vertex cover is a vertex cover of minimum cardinality. The celebrated theorem of K\H{o}nig gives the relationship between minimum vertex covers and maximum matchings.
\begin{theorem}[K\H{o}nig \cite{K31}]\label{thm:Konig}
In any bipartite graph, the size of a maximum matching is equal to
the size of a minimum vertex cover.
\end{theorem}

\begin{proposition}\label{non-adj}
     For $r\ge 3$, every pair of non-adjacent vertices in a $K_r$-saturated graph $G$ has at least $r-2$ common neighbors.
\end{proposition}
\begin{proof}
    Let $x,y\in V(G)$ be a pair of non-adjacent vertices. As $G$ is $K_r$-saturated, we have that  $G\cup\{xy\}$ produces a new $K_r$. Thus, $x$ and $y$ have at least $r-2$ common neighbors. 
\end{proof}

\section{Proof of Theorem \ref{thm:rs}} \label{SEC:prelim}

The graphs $K_{r-3}\vee K_{2,n-r+1}$, $G(n,r,s)$ and the graph family $\mathcal{H}(n,s)$ imply that 
\begin{equation*}     
     sat(n,r,s)\leq\left\{\begin{array}{cl}{(r-1)n-\frac{r}{2}(r-1)-1,}&{\quad\mathrm{if}~s=r-1;}\\{(r-1)n + (s-r)^2 - \frac{1}{2}(r+2)(r-3) - 5,}&{\quad\mathrm{if}~s>r-1.}\\\end{array}\right.
\end{equation*} 
It remains to prove the lower bound.

We first prove the case $sat(n,3,2)\geq 2n-4$. Let $G$ be an $n$-vertex $K_3$-saturated graph with $\nu(G)=2$ and $2n-4$ edges. Since $\nu(G)=2$, $G$ contains no odd cycle of length $\ge 7$. Otherwise the matching number would be at least $3$. 
If $G$ contained a $C_5$, then because $n\geq 63(s-2)^2+r^2=9$ and $G$ is $K_3$-saturated, there would exist a vertex outside the $C_5$ with a neighbor in $C_5$.
Then one can find a matching of size $3$ by taking two disjoint edges inside the $C_5$ together with an edge from the external vertex to one of its neighbor in $C_5$, contradicting $\nu(G)=2$. 
Hence $G$ contains no odd cycle and thus $G$ is bipartite.

A $K_3$-saturated bipartite graph must be complete bipartite. Otherwise, if some edge between the two partite sets were missing, adding it could not create a triangle. 
Therefore $G\cong K_{a,b}$ with $a+b=n$.

Now $\nu(G)=\min\{a,b\}=2$, so $\min\{a,b\}=2$. Consequently $e(G)=ab=2(n-2)=2n-4$. 
Hence $\operatorname{sat}(n,3,2)\ge 2n-4$.
\subsection{The basic case:  $r=3$ and $s\geq r$}

Let $G$ be an $n$-vertex $K_3$-saturated graph with $\nu(G)=s$ and $sat(n,3,s)$ edges. 

\begin{claim}\label{lem:C_5}
   The graph $G$ contains at least one cycle of length five.
\end{claim}
\begin{proof}
    Suppose that $G$ does not contain any odd cycles. Then $G$ is a bipartite graph with two parts $A$ and $B$. Moreover, we claim that $G$ is a complete bipartite graph. 
    If $G$ is not a complete bipartite graph, then we let $a\in A$, $b\in B$ be a pair of non-adjacent vertices of $G$.   
    By Proposition \ref{non-adj}, the vertex pair $a b$ has a common neighbor, which contradicts our assumption that $A$ and $B$ form a partition of $V(G)$ such that $G$ is bipartite.
    Thus $G$ is a complete bipartite graph.
    Since the matching number  of $G$ is $s$, we have $\min\{A,B\}\geq s$. Then 
    $$ e(G)\geq s(n-s)>e(G(3,s)),$$ 
    which contradicts our choice of $G$.

    Therefore $G$ contains an odd cycle of length at least five. Let $C_0$ be a shortest odd cycle, and write $V(C_0)=\{v_1,\dots,v_{2k+1}\}$. If $|C_0|=5$, then we are done. Thus we may assume $2k+1\ge 7$. Consider the vertices $v_1$ and $v_{k+1}$. If $v_1\sim v_{k+1}$, then $C_0$ is not a shortest odd cycle, a contradiction. Hence $v_1$ and $v_{k+1}$ are non-adjacent. By Proposition \ref{non-adj}, they have a common neighbor $v_c$. If $v_c\notin V(C_0)$, then either $\{v_1,\dots,v_{k+1}, v_c\}$ or 
    $\{v_{k+1},\dots,v_{2k+1}, v_1, v_c\}$ contains a shorter odd cycle. If $v_c\in V(C_0)$, then there is a shorter odd cycle inside $C_0$. Both cases contradict the choice of $C_0$. Thus the shortest odd cycle has length five.
\end{proof}

Let $C_5$ be a cycle of length five in $G$ with vertex set $\{v_1,v_2,v_3,v_4,v_5\}$. Since $G$ is triangle-free, every vertex in $V(G)\backslash V(C_5)$ is adjacent to at most two vertices of $C_5$. We partition the remaining vertices of $G$ as follows (see Figure \ref{fig:claim:basic}).

\begin{enumerate}[label=(\roman*)]
            \item $A=\{v:d_{C_5}(v)=2,v\notin V(C_5)\}$.
            \item $B=\{v:d_{C_5}(v)=1,v\notin V(C_5)\}$.  
             \item $C=\{v:d_{C_5}(v)=0,v\notin V(C_5)\}$.  
 \end{enumerate}
Clearly, $A$, $B$, $C$ are pairwise disjoint sets and $|A|+|B|+|C|=n-5$.
We can further partition $A$ into subsets
$$A_{i,j}=\{v:v\sim v_i, v\sim v_j, v\in A\},$$ 
where $1\leq i< j\leq 5$.
As $G$ is triangle-free, each vertex in $A$ must be adjacent to two non-consecutive vertices of $C_5$. Thus
$$A=A_{1,3}\cup A_{1,4}\cup A_{2,4}\cup A_{2,5}\cup A_{3,5}.$$
Similarly, we partition $B$ into subsets
$$B_i=\{v: v\in B,v\sim v_i \}.$$
Note that each of $A_{i,j}$ and $B_i$ is an independent set. Meanwhile, $G[V(C_5)]$ has five edges and 
\begin{equation}
e(A,V(C_5))=2|A|.\tag{1}
\end{equation}
For a vertex $v\in B$, we may assume $v\sim v_1$.  By Proposition \ref{non-adj}, the vertex pair $v$ and $v_3$ has at least one common neighbor $z_1$, and the pair $v$ and $v_4$ has one common neighbor  $z_2$. Since $G$ is  $K_3$-free, $z_1$ and $z_2$ are distinct.  As $C$ consists of vertices adjacent to no vertex of $C_5$, both $z_1$ and $z_2$ lie in $A$ or $B$. Thus $d_{A\cup B}(v)\geq 2$ and  
\begin{equation}
e(B)+e(A,B)+e(B,V(C_5))\geq 2|B|.\tag{2}
\end{equation}

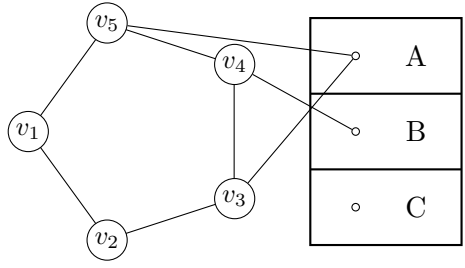
\begin{figure}[htbp]
    \centering
   \begin{tikzpicture}[
    vertex/.style={circle, draw, inner sep=1.5pt, fill=white, font=\small},
    port/.style={circle, draw, inner sep=1pt, fill=white},
]

\node[regular polygon, regular polygon sides=5, minimum size=3cm, 
      draw, shape border rotate=90] (pent) {};

\node[vertex] (v1) at (pent.corner 1) {$v_1$};
\node[vertex] (v2) at (pent.corner 2) {$v_2$};
\node[vertex] (v3) at (pent.corner 3) {$v_3$};
\node[vertex] (v4) at (pent.corner 4) {$v_4$};
\node[vertex] (v5) at (pent.corner 5) {$v_5$};

\def\boxwidth{2.0cm}
\def\boxheight{3.0cm}

\coordinate (boxSW) at ($(pent.east) + (1.0cm, -\boxheight/2)$);
\coordinate (boxNE) at ($(boxSW) + (\boxwidth, \boxheight)$);

\draw[thick] (boxSW) rectangle (boxNE);

\coordinate (line1) at ($(boxSW)!1/3!(boxNE)$); 
\coordinate (line2) at ($(boxSW)!2/3!(boxNE)$); 

\draw[thick] (boxSW |- line1) -- (boxNE |- line1);
\draw[thick] (boxSW |- line2) -- (boxNE |- line2);

\coordinate (centerA) at ($(boxNE)!1/6!(boxSW)$); 
\coordinate (centerB) at ($(boxSW)!1/2!(boxNE)$); 
\coordinate (centerC) at ($(boxSW)!1/6!(boxNE)$); 

\coordinate (portLine) at ($(boxSW)!0.3!(boxNE)$);   
\coordinate (letterLine) at ($(boxSW)!0.7!(boxNE)$); 

\coordinate (portA) at (portLine |- centerA);
\coordinate (portB) at (portLine |- centerB);
\coordinate (portC) at (portLine |- centerC);

\node[port] (pA) at (portA) {};
\node[port] (pB) at (portB) {};
\node[port] (pC) at (portC) {};

\coordinate (letterA) at (letterLine |- centerA);
\coordinate (letterB) at (letterLine |- centerB);
\coordinate (letterC) at (letterLine |- centerC);

\node at (letterA) {A};
\node at (letterB) {B};
\node at (letterC) {C};

\draw (v3) -- (pA); 
\draw (v5) -- (pA); 
\draw (v4) -- (pB); 
\end{tikzpicture}
\caption{Partition of V(G)}
    \label{fig:claim:basic}
\end{figure}

\begin{claim}\label{cl:C}
 $|C|\leq (s-3)^2$.
\end{claim}
\begin{proof}
Suppose to the contrary that $|C|>(s-3)^2$. Since any $v\in C$ has no neighbor in $C_5$, by Proposition \ref{non-adj}, $v$ has at least one common neighbor with each vertex of $C_5$. For $1\leq i\leq 5$, denote a common neighbor of $v$ and $v_i$ by $x_i$. Then $x_i\in A\cup B$. Since $A$ consists of vertices adjacent to exactly two vertices of $C_5$, and $B$ consists of vertices adjacent to
exactly one vertex of $C_5$, the set $\{x_1,x_2,x_3,x_4,x_5\}$ contains at least three distinct vertices. This implies that each $v\in C$ has at least three neighbors in $A\cup B$ and 
\begin{equation}
e(C,A\cup B)\geq 3C.\tag{3}
\end{equation}
Combining this with (1) and (2), we have 
\begin{align*}
e(G)
&\geq e(G[V(C_5)]) +e(A,V(C_5))+e(B)+e(A,B)+e(B,V(C_5))+e(C,A\cup B)\\
&\geq 5+2|A|+2|B|+3|C|\\
&=2n-5+|C|\\
&>2n+(s-3)^2-5\\
&=e(G(n,3,s)),
\end{align*}
a contradiction. Thus $|C|\leq (s-3)^2$.
\end{proof}

\begin{claim}\label{cl:B}
 $|B|\leq 52(s-3)^2$.
\end{claim}
\begin{proof}
Let $v$ be a vertex of $B_i$, where $1\leq i \leq 5$. Clearly, $v_i\in N(v)$.  We say that $v$ is ``good'' if both the common neighbors of $v$, $v_{i+2}$ lie in $B$ and the common neighbors of $v$, $v_{i+3}$ lie in $B$; here the indices are taken modulo $5$. Otherwise, we call it ``bad''. Since each bad vertex has at least one neighbor in $A$, if the cardinality of bad vertices is larger than $2(s-3)^2$,
then the inequality (2) can be modified by
$$e(B)+e(A,B)+e(B,V(C_5))> 2|B|+(s-3)^2.$$
Combining this with (1) and (3), we have
\begin{align*}
e(G)
&\geq e(G[V(C_5)]) +e(A,V(C_5))+e(B)+e(A,B)+e(B,V(C_5))+e(C,A\cup B)\\
&> 5+2|A|+2|B|+(s-3)^2+3|C|\\
&\geq 2n+(s-3)^2-5\\
&=e(G(n,3,s)),
\end{align*}
a contradiction. Thus the cardinality of bad vertices is at most $2(s-3)^2$. 

Let $B_g$ be the set of good vertices and $B_i'=B_g\cap B_i$. Assume that the cardinality of good vertices is more than $50(s-3)^2$. Then there exists a $B_i'$, say $B_1'$, has more than $10(s-3)^2$ vertices. 
As each vertex of $B_1'$ has at least a neighbor in $B_3$, we have 
$$\sum_{v\in B_3}d_{A\cup B}(v)\geq |B_1'|> 10(s-3)^2.$$ 
Recall that every vertex of $B$ has degree at least $2$ in $A\cup B$. If $|B_3|\leq 4(s-3)^2$, then
\begin{align*}
e(B) + e(A, B)\geq \frac{1}{2}\sum_{v\in B}d_{A\cup B}(v)
&=
\frac{1}{2}\sum_{v\in B_3}d_{A\cup B}(v)+\frac{1}{2}\sum_{v\in B\backslash B_3}d_{A\cup B}(v).\\
&> 5(s-3)^2+|B\backslash B_3|\\
&\geq |B|+(s-3)^2.
\end{align*}
Again we have $e(B)+e(A,B)+e(B,V(C_5))> 2|B|+(s-3)^2$, and hence $e(G)>e(G(n,3,s))$, a contradiction. Therefore $|B_3|> 4(s-3)^2$ and $|B_3'|> 2(s-3)^2$, since the number of bad vertices is at most $2(s-3)^2$.

Let $X=\{x_1,x_2,\ldots,x_{2(s-3)^2+1}\}$ be a subset of $B_3'$.
Then at most $8(s-3)^2$ vertices of $B_1'$ have a neighbor in $X$. Otherwise, $\sum_{v\in X}d_{A\cup B}(v) > 8(s-3)^2$, and the same estimate gives $e(B)+e(A,B)> |B|+(s-3)^2$, which again implies $e(G)>e(G(n,3,s))$, a contradiction. Thus we can choose a set $Y=\{y_1,y_2,\ldots,y_{2(s-3)^2+1}\}\subseteq B_1'$ such that $e(X,Y)=0$.
By Proposition \ref{non-adj} and by the definition of good vertices, for each $1\leq i\leq 2(s-3)^2+1$, the vertex $x_i$ has a neighbor in $B_1\backslash Y$ and a neighbor in $B_5$, while $y_i$ has a neighbor in $B_3\backslash X$ and a neighbor in $B_4$. Moreover, $x_i$ and $y_i$ have a common neighbor in $V(G)\backslash V(C_5)$. It follows that
$d_{A\cup B}(x_i)+d_{A\cup B}(y_i)\geq 5$, and hence $\sum_{v\in X\cup Y}d_{A\cup B}(v) > 10(s-3)^2+5$. Again we obtain $e(B)+e(A,B)> |B|+(s-3)^2$ and thus $e(G)>e(G(n,3,s))$, a contradiction. Thus the number of good vertices is at most $50(s-3)^2$.
\end{proof}
    
Claim \ref{cl:C} and Claim \ref{cl:B} imply that
\begin{equation*}
|A|\geq n-53(s-3)^2.
\end{equation*}

Recall that $n> 63(s-2)^2$.
Thus $|A|>10(s-3)^2 $ and there exists an $A_{i,j}$, say $A_{1,3}$, has more than $2(s-3)^2$ vertices. 

\begin{claim}\label{cl:ABC}
 $B=C=A_{2,4}=A_{2,5}=\emptyset$.   
\end{claim}
\begin{proof}
Suppose that there is a vertex $v\in A_{2,4}$. Then by Proposition \ref{non-adj}, for each vertex $u\in A_{1,3}$, either $u$ is adjacent to $v$ or $u$ and $v$ have a common neighbor in $V(G)\backslash V(C_5)$. So $d_{V(G)\backslash V(C_5)}(u)\geq 1$ and $\sum_{u\in A_{1,3}}d_{V(G)\backslash V(C_5)}(u)\geq |A_{1,3}|>2(s-3)^2$. Combining this with (1),(2) and (3), we have  
\begin{align*}
e(G)
&\geq e(V(C_5))+e(A,V(C_5))+\frac{1}{2}\sum_{u\in A_{1,3}}d_{V(G)\backslash V(C_5)}(u)\\
&~~~+e(B)+e(A,B)+e(B,V(C_5))+e(C,A\cup B)\\
&>2n-5+(s-3)^2\\
&=e(G(n,3,s)),
\end{align*}
a contradiction. Thus $A_{2,4}=\emptyset$. The same argument with $A_{2,5}$ in place of $A_{2,4}$ gives $A_{2,5}=\emptyset$. Moreover, applying the preceding edge-counting argument to a vertex of $B_2$, $B_4$, $B_5$, or $C$ gives $B_2=B_4=B_5=C=\emptyset$.

It remains to show that $B_1=B_3=\emptyset$. Suppose that there is a vertex $v\in B_1$. By Proposition \ref{non-adj}, the vertices $v$ and $v_4$ have a common neighbor, say $z$. Since $A_{2,4}=B_4=\emptyset$, we have $z\in A_{1,4}$. Then $v,v_1,z$ form a triangle, a contradiction. Hence $B_1=\emptyset$, and the same argument gives $B_3=\emptyset$.
\end{proof}

Claim \ref{cl:ABC} implies that $V(G)=V(C_5)\cup A_{1,3}\cup A_{3,5}\cup A_{1,4}$. Moreover, $e(A_{1,3},A_{3,5})=e(A_{1,3},A_{1,4})=0$, since $G$ is triangle-free.  

\begin{claim}\label{claim:compbi}
If both $A_{1,4}\neq \emptyset$ and $A_{3,5}\neq \emptyset$, then $G[A_{1,4}\cup A_{3,5}]$ is a complete bipartite graph.  
\end{claim}
\begin{proof}
 For each $u\in A_{1,4}$ and $v\in A_{3,5}$, if $u$ and $v$ had a common neighbor, then $G$ would contain a triangle. Thus, by Proposition \ref{non-adj}, $u$ must be adjacent to $v$. Therefore $G[A_{1,4}\cup A_{3,5}]$ is a complete bipartite graph.   
\end{proof}
Clearly, if at least one of $A_{1,4}$ and $A_{3,5}$ is an empty set, then the matching number of $G$ is exactly $3$.
By Claim \ref{claim:compbi}, when both $A_{1,4}\neq \emptyset$ and $A_{3,5}\neq \emptyset$, the graph contains a matching of size at least $4$, so this case can occur only when $s>3$. Hence, for $s=3$, at least one of $A_{1,4}$ and $A_{3,5}$ is empty. It follows that $G\in \mathcal{H}(n,3)$.  
When $\nu(G)=s>3$, both $A_{1,4}$ and $A_{3,5}$ are nonempty. By Claim \ref{claim:compbi}, the graph $G$ is a blow-up of $C_5$ (see Figure \ref{fig:structure}). More precisely, 
$$G=C_5(1,|A_{1,3}|+1,1,|A_{3,5}|+1,|A_{1,4}|+1).$$
\begin{figure}[htbp]
    \centering

    \begin{subfigure}{0.45\textwidth}
        \centering
        \begin{tikzpicture}[scale=1.0, every node/.style={circle, draw, inner sep=1pt, font=\small}]
            \node (v1) at (90:1.5) {$v_1$};
            \node (v2) at (18:1.5) {$v_2$};
            \node (v3) at (-54:1.5) {$v_3$};
            \node (v4) at (-126:1.5) {$v_4$};
            \node (v5) at (-198:1.5) {$v_5$};

            \node[ellipse, draw, minimum width=1cm, minimum height=0.6cm] (A13) at (18:3) {$A_{13}$};
            \node[ellipse, draw, minimum width=1cm, minimum height=0.6cm] (A35) at (-135:3) {$A_{35}$};
\node[ellipse, draw, minimum width=1cm, minimum height=0.6cm] (A14) at (-198:3) {$A_{14}$};
            \draw (v1) -- (v2) -- (v3) -- (v4) -- (v5) -- (v1);

            \draw (A13) -- (v1);
            \draw (A13) -- (v3);
            \draw (A35) -- (v3);
            \draw (A35) -- (v5);
\draw (A14) -- (A35);
 \draw (A14) -- (v4);
  \draw (A14) -- (v1);
        \end{tikzpicture}
     
        \label{fig:stru}
    \end{subfigure}
    \hfill

 \caption{The structure of $G$}
    \label{fig:structure}
\end{figure}
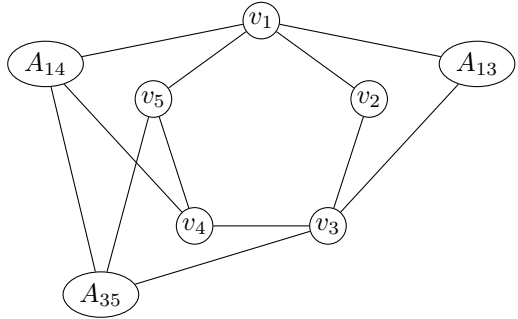
Assume that $|A_{1,4}|\leq s-4$. Then $G\backslash v_1$ is bipartite, and $A_{1,4}\cup \{v_3,v_5\}$ is a vertex cover of $G\backslash v_1$. By K\H{o}nig's theorem (Theorem \ref{thm:Konig}), $\nu(G\backslash v_1)\leq s-2$, and thus $\nu(G)\leq s-1$, a contradiction. Hence $|A_{1,4}|\geq s-3$, and similarly $|A_{3,5}|\geq s-3$. The extremal choice of $G$ then gives $|A_{1,4}|=|A_{3,5}|=s-3$. Therefore $G\cong G(n,3,s)$. This completes the proof. 

\subsection{The case for $r>3$}
In this subsection, we prove the case for $r\geq 4$. Let
\begin{equation*}
     f(n,r,s)=\left\{\begin{array}{cl}{(r-1)n-\frac{r}{2}(r-1)-1,}&{\quad\mathrm{if}~s=r-1;}\\{(r-1)n + (s-r)^2 - \frac{1}{2}(r+2)(r-3) - 5,}&{\quad\mathrm{if}~s>r-1.}\\\end{array}\right.
\end{equation*}
A direct calculation gives 
\begin{equation*}
f(n,r,s)-f(n-1,r-1,s-1)=n-1. \tag{4}
\end{equation*}
\begin{lemma}\label{lem:n-1degree}
For $s+1\geq r\geq 4$, let $G$ be an $n$-vertex $K_r$-saturated graph with $\nu(G)=s$ and $e(G)=sat(n,r,s)$. Then $G$ has a vertex of degree $n-1$.
\end{lemma}
    \begin{proof}
Suppose to the contrary that $d(v)< n-1$ for all $v\in V(G)$. 
Since the extremal constructions described above have $f(n,r,s)$ edges, we have $e(G)\le f(n,r,s)$.
We first discuss the case when $r=4$. By Theorem \ref{thm:4degree}, if $\delta(G)=4$ then $e(G)>f(n,4,s)$, which is a contradiction. So we may assume $\delta(G)\ge 5$.
Moreover, we have $\delta(G)\le 6$, since if $\delta(G)\ge 7$ then $e(G)> f(n,4,s)$.
Let $v_0$ be a vertex of minimum degree and denote the neighborhood of $v_0$ as $X$ and the other vertices as $Y$.  By Proposition \ref{non-adj}, each vertex in $Y$ has at least $2$ common neighbors with $v_0$. So we have $d_X(v)\ge 2$ for any $v\in Y$. Then
\begin{align*}
e(G)&=e(X,Y)+e(X)+e(Y)\\
    &\geq 2|Y|+\frac{1}{2}(\delta(G)-2)|Y|\\
    &=2(n-\delta(G))+\frac{1}{2}(\delta(G)-2)(n-\delta(G))\\
    &>f(n,4,s),
\end{align*}
where the last inequality follows from $5\leq \delta(G)\leq 6$ and a simple calculation.
Again it is a contradiction.

Next, we discuss the cases when $r\ge 5$. By Theorem \ref{thm:rdegree}, $\delta(G)\ge 2(r-2)$. Since $e(G)\leq f(n,r,s)$, we have $\delta(G)\le 2(r-1)$. Let $v\in V(G)$ be a vertex of minimum degree. Similarly, by Proposition \ref{non-adj},
each vertex in $V(G)\backslash N(v)$ has at least $r-2$  neighbors in $N(v)$. Thus 
\begin{align*}
e(G)&=e(N(v),V(G)\backslash N(v))+e(N(v))+e(V(G)\backslash N(v))\\
    &\geq (r-2)(n-\delta(G))+\frac{1}{2}(\delta(G)-r+2)(n-\delta(G))\\
    &>f(n,r,s),
\end{align*}
which is a contradiction.  
\end{proof}

Now we prove Theorem \ref{thm:rs} by induction on $r$. Recall that we have already proven the case $r=3$. Suppose Theorem \ref{thm:rs} holds for $r=i$, where $i\geq 3$. Let $G$ 
be an $n$-vertex $K_{i+1}$-saturated graph with $\nu(G)=s$ and $sat(n,i+1,s)$ edges. By Lemma \ref{lem:n-1degree}, we can pick a vertex $v$ of degree $n-1$. Consider the graph $G'=G\backslash v$. We have $\nu(G')=s-1$ and $G'$ is $K_i$-saturated. By the induction hypothesis, 
 $$e(G')\ge f(n-1,i,s-1).$$ 
 By (4), we have
 \begin{align*}
 e(G)\ge e(G')+n-1\geq f(n,i+1,s).
\end{align*}
Thus $sat(n,r,s)= f(n,r,s)$. Moreover, since each step of the induction has a vertex adjacent to all the remaining vertices, the extremal structure is obtained by joining $K_{r-3}$ to the corresponding extremal graph in the case $r=3$. Hence $G\cong K_{r-3}\vee K_{2,n-r+1}$ for $s=r-1$, $G\in \mathcal{H}(n,s)$ for $s=r$, and $G\cong G(n,r,s)$ for $s> r$. This completes the proof of Theorem \ref{thm:rs}.  

\section{Concluding Remark}
In this paper, we focus on the saturation problems for graphs with fixed matching number. We determine the saturation number $sat(n,r,s)$ for $s\leq 2+\sqrt{\frac{n-r^2}{63}}$. In fact, the coefficient $\sqrt{1/63}$ can be improved by a more careful calculation. A natural question is whether $sat(n,r,s)$ still equals $e(G(n,r,s))$ when $s=\Omega(\sqrt{n})$. The answer is negative. For example, when $k\geq 100$ and $n=4k+1$, the graph $C_5(1,k,k,k,k)$ is a $K_3$-saturated graph with matching number $2k$, but it has fewer edges than $G(n,3,2k)$. Thus we propose the following problem.

\begin{problem}
 Determine $sat(n,r,s)$ for $s=\Omega(\sqrt{n})$.   
\end{problem}

{\noindent \bf Acknowledgments}. 
The authors are grateful to Jialin He, Xizhi Liu and Jun Gao for their helpful discussions.

\bibliographystyle{abbrv}
\bibliography{Saturation}

\end{document}